\documentclass{amsart}

\usepackage{graphicx}
\usepackage{amsfonts}
\usepackage{amssymb}
\usepackage{amsmath}
\usepackage{amsxtra}
\usepackage{latexsym}

\def\d{\delta}

\def\B{\mathcal B}
\def\A{\mathcal A}
\def\a{\alpha}

\def\d{\delta}
\def\la{\lambda}

\newtheorem{theorem}{Theorem}[section]
\newtheorem{lemma}[theorem]{Lemma}
\newtheorem{Assumption}{Assumption}

\newtheorem{proposition}{Proposition}[section]

\theoremstyle{definition}

\theoremstyle{remark}

\numberwithin{equation}{section}

\begin{document}

\title[]{A general convergence analysis on inexact Newton method for
nonlinear inverse problems}

\author{Qinian Jin}
\address{Department of Mathematics, Virginia Tech, Blacksburg, VA
24061} \email{qnjin@math.vt.edu}


\date{August 1, 2010}



\begin{abstract}
We consider the inexact Newton methods
$$
x_{n+1}^\d=x_n^\d-g_{\a_n}\left(F'(x_n^\d)^* F'(x_n^\d)\right) F'(x_n^\d)^* \left(F(x_n^\d)-y^\d\right)
$$
for solving nonlinear ill-posed inverse problems $F(x)=y$ using the only available noise data
$y^\d$ satisfying $\|y^\d-y\|\le \d$ with a given small noise level $\d>0$. We terminate the iteration
by the discrepancy principle
$$
\|F(x_{n_\d}^\d)-y^\d\|\le \tau \d<\|F(x_n^\d)-y^\d\|, \qquad 0\le n<n_\d
$$
with a given number $\tau>1$. Under certain conditions on $\{\a_n\}$ and $F$, we prove for a large class
of spectral filter functions $\{g_\a\}$ the convergence of $x_{n_\d}^\d$ to a true solution as
$\d\rightarrow 0$. Moreover, we derive the order optimal rates of convergence
when certain H\"{o}lder source conditions hold. Numerical examples are given to test the
theoretical results.
\end{abstract}

\maketitle

\def\A{\mathcal A}
\def\B{\mathcal B}
\def\a{\alpha}
\def\d{\delta}
\def\la{\lambda}

\section{\bf Introduction}
\setcounter{equation}{0}

In this paper we consider the nonlinear equations
\begin{equation}\label{1}
F(x)=y,
\end{equation}
arising from nonlinear inverse problems, where $F: D(F)\subset
X\mapsto Y$ is a nonlinear Fr\'{e}chet differentiable operator
between two Hilbert spaces $X$ and $Y$ whose norms and inner
products are denoted as $\|\cdot\|$ and $(\cdot, \cdot)$
respectively. We assume that (\ref{1}) has a solution $x^\dag$
in the domain $D(F)$ of $F$, i.e. $F(x^\dag)=y$. We use
$F'(x)$ to denote the Fr\'{e}chet derivative of $F$ at $x\in D(F)$
and $F'(x)^*$ the adjoint of $F'(x)$.
A characteristic property of such problems is
their ill-posedness in the sense that their solutions do not
depend continuously on the data. Since the right hand side $y$ is usually
obtained by measurement, the only available data is a noise
$y^\d$ satisfying
\begin{equation}\label{1.2}
\|y^\delta-y\|\le \delta
\end{equation}
with a given small noise level $\delta> 0$. Due to the
ill-posedness, it is challenging to produce
from $y^\d$ a stable approximate solution to $x^\dag$ and the regularization
techniques must be taken into account.

Many regularization methods have been considered for solving (\ref{1}) in the last two
decades. Tikhonov regularization is one of the well-known methods that have been
studied extensively in the literature.  Due to the straightforward implementation,
iterative methods are also attractive for solving nonlinear inverse problems.
In this paper we will consider a class of inexact Newton methods. To motivate,
let $x_n^\d$ be a current iterate. We may approximate $F(x)$ by its linearization around $x_n^\d$,
i.e. $F(x)\approx F(x_n^\d) +F'(x_n^\d) (x-x_n^\d)$. Thus, instead of (\ref{1}) we have
the approximate equation
\begin{equation}\label{NT}
F'(x_n^\d) (x-x_n^\d) =y^\d-F(x_n^\d).
\end{equation}
If $F'(x_n^\d)$ is invertible, the usual Newton method defines the next iterate by
solving (\ref{NT}) for $x$. Computing the exact solution of (\ref{NT}) however can be expensive
in general even the problem is well-posed. Thus, one might prefer
to compute some approximate solution at certain accuracy and use it as the next iterate.
This motivates the inexact Newton methods in \cite{DES82} where for well-posed problems
the convergence was carried out when the next computed iterate $x_{n+1}^\d$ satisfies
\begin{equation}\label{m2}
\|F(x_n^\d)-y^\d+F'(x_n^\d)(x_{n+1}^\d-x_n^\d)\|\le \mu_n \|F(x_n^\d)-y^\d\|
\end{equation}
at each step with the forcing terms $\mu_n\in (0,1)$ being uniformly bounded below $1$.
For nonlinear ill-posed inverse problems, $F'(x_n^\d)$ in general is not invertible and
(\ref{NT}) usually is ill-posed. Therefore one should use the regularization methods
to solve (\ref{NT}) approximately. Let $\{g_\a\}$ be a family of spectral filter functions.
We can apply the linear regularization method defined by $\{g_\a\}$
to (\ref{NT}) to produce the next iterate.
This leads to the following inexact Newton method
\begin{equation}\label{2}
x_{n+1}^\d =x_n^\d-g_{\alpha_n} \left(F'(x_n^\d)^*F'(x_n^\d)\right) F'(x_n^\d)^*
\left(F(x_n^\d)-y^\d\right),
\end{equation}
where $x_0^\d:=x_0\in D(F)$ is an initial guess of $x^\dag$ and  $\{\a_n\}$ is a sequence of positive numbers.
By taking $g_\a$ to be various functions, (\ref{2}) then produces the
nonlinear Landweber iteration \cite{HNS96}, the Levenberg-Marquardt method \cite{H97,Jin10a},
the exponential Euler iteration \cite{HHO09}, and the
first-stage Runge-Kutta type regularization \cite{PB10}.

In this paper we will consider the inexact Newton method (\ref{2}) in a unified way
by assuming that $\{\a_n\}$ is an a priori given sequence of positive numbers
with suitable properties. We will terminate the iteration by the discrepancy principle
\begin{equation}\label{DP}
\|F(x_{n_\d}^\d)-y^\d\|\le \tau \d <\|F(x_n^\d)-y^\d\|, \quad 0\le n<n_\d
\end{equation}
with a given number $\tau>1$ and consider the approximation property
of $x_{n_\d}^\d$ to $x^\dag$ as $\d\rightarrow 0$. For a large class of
spectral filter functions $\{g_\a\}$ we will establish the convergence of $x_{n_\d}^\d$
to $x^\dag$ as $\d\rightarrow 0$ and derive  the order optimal convergence rates
for the method defined by (\ref{2}) and (\ref{DP}). Our work not only reproduces those
known results in \cite{HNS96,Jin10a,HHO09,PB10} but also presents new convergence results and new methods.
Furthermore, our convergence analysis provides new insights into the feature of the inexact
Newton regularization methods.

In the definition of the inexact Newton method, one may determine the sequence $\{\a_n\}$ adaptively
during computation. In \cite{H97} the Levenberg-Marquardt scheme was considered
with $\{\a_n\}$ chosen adaptively so that (\ref{m2}) holds and
the discrepancy principle was used to terminate the iteration. The order optimal
convergence rates were derived recently in \cite{H2010}.
The general methods (\ref{2}) with $\{\a_n\}$ chosen adaptively to satisfy
(\ref{m2}) were considered later in \cite{R99,LR2010}, but only suboptimal
convergence rates were derived in \cite{R01} and the convergence analysis
is far from complete. The methods of the present paper is essentially different in that
the sequence $\{\a_n\}$ is given in an a priori way which has the advantage of saving computational work.
We hope, however, the work of the present paper can provide better understanding on the methods with
$\{\a_n\}$ chosen adaptively.

This paper is organized as follows. In Section 2 we first formulate the conditions on
$\{\a_n\}$, $\{g_\a\}$ and $F$, and state the main results on the convergence and
rates of convergence for the methods defined by (\ref{2}) and (\ref{DP}), we then
give several examples of iteration methods that fit into the framework (\ref{2}). In Section 3
we prove some crucial inequalities which is frequently used in the convergence analysis.
In Section 4 we derive the order optimal convergence rate result when $x_0-x^\dag$ satisfies
certain source conditions. In Section 5 we show the convergence property
without assuming any source conditions on $x_0-x^\dag$. Finally in Section 5 we present numerical examples
to test the theoretical results.

\section{\bf Main results}
\setcounter{equation}{0}

In order to carry out the convergence analysis on the method
defined by (\ref{2}) and (\ref{DP}), we need to impose
suitable conditions on $\{\a_n\}$, $\{g_\a\}$ and $F$.
For the sequence $\{\a_n\}$ of positive numbers, we set
\begin{equation}\label{a1}
s_{-1}=0, \qquad s_n:=\sum_{j=0}^n \frac{1}{\a_j}, \qquad n=0, 1,\cdots.
\end{equation}
We will assume that there are constants $c_0>1$ and $c_1>0$ such that
\begin{equation}\label{60}
\lim_{n \rightarrow \infty} s_n =\infty, \quad
s_{n+1}\le c_0 s_n \quad \mbox{and} \quad 0<\a_n \le c_1, \quad n=0, 1,\cdots.
\end{equation}
For the spectral filter functions $\{g_\a\}$, we will assume the following two conditions,
where ${\mathbb C}$ denotes the complex plane.

\begin{Assumption}\label{A4}
For each $\a>0$, the function
$$
\varphi_\a(\lambda):=g_\a(\lambda)-\frac{1}{\a+\lambda}
$$
extends to a complex analytic
function defined on a domain $D_\a\subset {\mathbb C}$ such that
$[0, 1]\subset D_\a$, and there is a contour $\Gamma_\a\subset
D_\a$ enclosing $[0, 1]$ such that
\begin{equation}\label{2.6}
|z|\ge \frac{1}{2}\a \quad \mbox{and} \quad \frac{|z|+\la}{|z-\la|}\le b_0, \qquad \forall z \in
\Gamma_\a, \, \a>0 \mbox{ and } \la \in [0,1],
\end{equation}
where $b_0$ is a constant independent of $\a>0$.
Moreover, there is a constant $b_1$ such that
\begin{equation}\label{2.8}
\int_{\Gamma_\a} \left|\varphi_\a(z)\right| |d z|\le b_1
\end{equation}
for all $0<\a\le c_1$.
\end{Assumption}

\begin{Assumption}\label{A5}
Let $\{\a_n\}$ be a sequence of positive numbers,
let $\{s_n\}$ be defined by (\ref{a1}). There is a constant $b_2>0$ such that
\begin{align}
0\le \la^\nu \prod_{k=j}^n r_{\a_k}(\la)&\le (s_n-s_{j-1})^{-\nu},  \label{g1}\\
0\le \la^\nu g_{\a_j}(\la) \prod_{k=j+1}^n r_{\a_k}(\la)&\le b_2
\frac{1}{\a_j} (s_n-s_{j-1})^{-\nu} \label{g2}
\end{align}
for $0\le \nu\le 1$,  $0\le \la \le 1$ and $j=0, 1, \cdots, n$, where
$r_\a(\lambda):=1-\la g_\a(\la)$ is the residual function.
\end{Assumption}

By using the spectral integrals for self-adjoint operators, it follows easily from
(\ref{2.6}) in Assumption \ref{A4} that for any bounded
linear operator $A$ with $\|A\|\le 1$ there holds
\begin{equation}\label{200}
\|(z I-A^*A)^{-1}(A^*A)^\nu\|\le \frac{b_0}{|z|^{1-\nu}}
\end{equation}
for $z \in \Gamma_\a$ and $0\le \nu\le 1$.
Moreover, since Assumption \ref{A4} implies $\varphi_\a(z)$ is analytic
in $D_\a$ for each $\a>0$, there holds the Riesz-Dunford formula (see \cite{BK04})
\begin{equation}\label{RD}
\varphi_\a(A^*A)=\frac{1}{2\pi i} \int_{\Gamma_\a} \varphi_\a(z) (z
I-A^*A)^{-1} d z
\end{equation}
for any linear operator $A$ satisfying $\|A\|\le 1$.

As a simple consequence of (\ref{g1}) in Assumption \ref{A5}, we have
for $0\le \nu\le 1$ and $\a>0$ that
\begin{equation}
0\le \la^{\nu} (\a+\la)^{-1} \prod_{k=j+1}^n r_{\a_k}(\la)
\le 2 \a^{\nu-1} \left(1+\a (s_n-s_j)\right)^{-\nu} \label{g3}
\end{equation}
for all $0\le \la\le 1$ and $j=0, 1, \cdots, n$, see \cite[Lemma 1]{JT2010}.

For the nonlinear operator $F$, we need the following condition
which has been verified in \cite{HNS96} for several nonlinear inverse problems.

\begin{Assumption}\label{A1}
(a) There exists $K_0\ge 0$ such that
\begin{equation}\label{A10}
F'(x)=R(x, \bar{x}) F'(\bar{x}) \quad \mbox{and} \quad \|I-R(x, \bar{x})\|\le K_0\|x-\bar{x}\|
\end{equation}
for all $x, \bar{x} \in B_\rho(x^\dag)\subset D(F)$.

(b) $F$ is properly scaled so that $\|F'(x) \|
\le \min\{1,\sqrt{\a_0}\}$ for all $x\in B_\rho(x^\dag)$.
\end{Assumption}

The condition (a) in Assumption \ref{A1} clearly implies that $\|F'(x)\|$ is uniformly bounded
over $B_\rho(x^\dag)$. Thus, by multiplying (\ref{1}) by a sufficiently small number, we may assume
that $F$ is properly scaled so that condition (b) in Assumption \ref{A1} is satisfied. A direct consequence
of Assumption \ref{A1} is the inequality
$$
\|F(x)-F(x^\dag)-F'(x^\dag)(x-x^\dag)\|\le \frac{1}{2} K_0\|x-x^\dag\| \|F'(x^\dag) (x-x^\dag)\|
$$
for all $x\in B_\rho(x^\dag)$, which will be frequently used in the convergence analysis.

Now we are ready to state the first main result concerning the rate of convergence of $x_{n_\d}^\d$ to
$x^\dag$ as $\d\rightarrow 0$ when
$e_0:=x_0-x^\dag$ satisfies the sourcewise condition
\begin{equation}\label{33}
x_0-x^\dag =(F'(x^\dag)^*F'(x^\dag))^\nu \omega
\end{equation}
for some $0<\nu\le 1/2$ and $\omega\in {\mathcal N}(F'(x^\dag))^\perp\subset X$,
where $n_\d$ is the integer determined by the discrepancy
principle (\ref{DP}) with $\tau>1$.

\begin{theorem}\label{T1}
Let $F$ satisfy Assumptions \ref{A1}, let $\{g_\a\}$
satisfy Assumptions \ref{A4} and \ref{A5},  and let $\{\a_n\}$ be
a sequence of positive numbers satisfying (\ref{60}).
If $x_0-x^\dag$ satisfies the source condition (\ref{33}) for some
$0<\nu\le 1/2$ and $\omega\in {\mathcal N}(F'(x^\dag))^\perp\subset X$
and if $K_0\|\omega\|$ is suitably small, then
$$
\|x_{n_\d}^\d-x^\dag\| \le C_\nu \|\omega\|^{1/(1+2\nu)} \d^{2\nu/(1+2\nu)}
$$
for the integer $n_\d$ determined by the discrepancy principle
(\ref{DP}) with $\tau>1$, where $C_\nu>0$ is a generic constant independent of $\d$
and $\|\omega\|$.
\end{theorem}

Theorem \ref{T1} shows that the method (\ref{1}) together
with the discrepancy principle (\ref{DP}) defines an order optimal regularization method
for each $0< \nu\le 1/2$. This result in particular reproduces the
corresponding ones in \cite{HNS96,Jin10a,HHO09,PB10} for various iterative methods
even with an improvement by relaxing $\tau>2$ to $\tau>1$.

Nevertheless, Theorem \ref{T1} does not provide the convergence of $x_{n_\d}^\d$ to $x^\dag$ as $\d\rightarrow 0$
if there is no source condition imposed on $x_0-x^\dag$. In the next main result we will show the convergence
of $x_{n_\d}^\d$ to $x^\dag$ as $\d\rightarrow 0$ if $\{\a_n\}$ is a geometric decreasing sequence, i.e.
\begin{equation}\label{7.19.1}
\a_n =\a_0 r^n, \qquad n=0,1, \cdots
\end{equation}
for some $\a_0>0$ and $0<r<1$, which is one of the most important cases in applications.

\begin{theorem}\label{T2}
Let $F$ satisfy Assumptions \ref{A1}, let $\{g_\a\}$
satisfy Assumptions \ref{A4} and \ref{A5},  and let $\{\a_n\}$ be
a sequence of positive numbers satisfying (\ref{7.19.1}).
If $x_0-x^\dag\in {\mathcal N}(F'(x^\dag))^\perp$
and $K_0\|x_0-x^\dag\|$ is suitably small, then
$$
\lim_{\d\rightarrow 0} x_{n_\d}^\d=x^\dag
$$
for the integer $n_\d$ determined by the discrepancy principle (\ref{DP}) with $\tau>1$.
\end{theorem}

Theorem \ref{T2} extends the corresponding result in \cite{Jin10a} for the Levenberg-Marquardt method
to a general class of methods given by (\ref{2}). The convergence result in Theorem \ref{T2}
still holds if (\ref{7.19.1}) is replaced by  the condition
\begin{equation}\label{7.29.10}
d_0 r^n\le \a_n \le d_1 r^n, \quad n=0,1,\cdots
\end{equation}
for some constants $0<d_0\le d_1<\infty$ and $0<r<1$. In fact, the proof of Theorem \ref{T2} given in
Section 5 requires only $\{\a_n\}$ to satisfy (\ref{60}) and (\ref{7.19.2}) which can be achieved if
$\{\a_n\}$ satisfies (\ref{7.29.10}). It would be interesting if such a convergence result can be proved for a general sequence
$\{\a_n\}$ satisfying (\ref{60}) only. This, however, remains open; new technique seems to be explored.

We conclude this section with several examples of the methods (\ref{2})
in which the spectral filter functions $\{g_\a\}$ have been shown in \cite{JT2010} to
satisfy Assumptions \ref{A4} and \ref{A5}:

(a)  We first consider for $\a>0$ the function
$g_\alpha$ given by
$$
g_\alpha(\lambda)
=\frac{(\a+\la)^N-\a^N}{\la (\a+\la)^N}
$$
where $N\ge 1$ is a fixed integer. This function arises from the
iterated Tikhonov regularization of order $N$ for linear ill-posed
problems. The corresponding method (\ref{2}) becomes
\begin{align*}
u_{n,0}&=x_n^\d,\\
u_{n, l+1}&=u_{n,l}-\left(\a_n I +F'(x_n^\d)^* F'(x_n^\d)\right)^{-1} F'(x_n^\d)^*
\left(F(x_n^\d)-y^\d -F'(x_n^\d)(x_n^\d-u_{n,l})\right), \\
&\qquad\qquad\qquad\qquad \qquad\qquad\qquad\qquad\qquad\qquad \quad l=0, \cdots, N-1,\\
x_{n+1}^\d&= u_{n, N}.
\end{align*}
When $N=1$, this is the Levenberg-Marquardt method (see \cite{H97,Jin10a}).

(b) We consider the method (\ref{2}) with $g_\a$ given by
$$
g_\alpha(\lambda)=\frac{1}{\lambda}\left(1-e^{-\lambda/\alpha}\right)
$$
which arises from the asymptotic regularization for linear
ill-posed problems. In this method, the iterative sequence $\{x_n^\d\}$ is
equivalently defined as $x_{n+1}^\d:=x(1/\a_n)$, where $x(t)$ is the unique solution of
the initial value problem
\begin{align*}
&\frac{d}{d t} x(t)= F'(x_n^\d)^* \left(y^\d-F(x_n^\d)+F'(x_n^\d)(x_n^\d-x(t))\right), \quad t>0,\\
&x(0)=x_n^\d.
\end{align*}
This is the so called exponential Euler iteration considered in \cite{HHO09}.

(c) For $0<\a\le 1$ consider the function $g_\alpha$ given by
\begin{equation}\label{g20}
g_\alpha(\lambda)=\sum_{l=0}^{[1/\alpha]-1}(1-\lambda)^l=\frac{1-(1-\la)^{[1/\a]}}{\la}
\end{equation}
which arises from the linear Landweber iteration, where $[1/\a]$ denotes the largest integer not
greater than $1/\a$. The method (\ref{2}) then becomes
\begin{align*}
u_{n,0}&=x_n^\d,\\
u_{n, l+1}&=u_{n,l}-F'(x_n^\d)^* \left(F(x_n^\d)-y^\d -F'(x_n^\d)(x_n^\d-u_{n,l})\right), \quad 0\le l\le [1/\a_n]-1,\\
x_{n+1}^\d&= u_{n, [1/\a_n]}.
\end{align*}
When $\a_n=1$ for all $n$, this method reduces
to the Landweber iteration in \cite{HNS96}.

(d) For $0<\a\le 1$ consider the function
$$
g_\alpha(\lambda)=\sum_{i=1}^{[1/\alpha]}(1+\lambda)^{-i}=\frac{1-(1+\la)^{-[1/\a]}}{\la}
$$
arising from the Lardy method for linear inverse problems. Then the method (\ref{2})
becomes
\begin{align*}
u_{n,0}&=x_n^\d,\\
u_{n, l+1}&=u_{n,l}-\left(I+F'(x_n^\d)^* F'(x_n^\d)\right)^{-1}
F'(x_n^\d)^* \left(F(x_n^\d)-y^\d -F'(x_n^\d)(x_n^\d-u_{n,l})\right),\\
&\qquad\qquad\qquad\qquad\qquad\qquad\qquad\qquad \qquad l=0, \cdots, [1/\a_n]-1,\\
x_{n+1}^\d&= u_{n, [1/\a_n]}.
\end{align*}
When $\a_n=1$ for all $n$, this is the so called first-stage Runge-Kutta type regularization considered in \cite{PB10}.

\section{\bf Some crucial inequalities}
\setcounter{equation}{0}

The following consequence of the above assumptions on $F$ and $\{g_\a\}$ plays a crucial
role in the convergence analysis.

\begin{lemma}\label{L10}
Let $\{g_\a\}$ satisfy Assumptions \ref{A4} and \ref{A5}, let
$F$ satisfy Assumption \ref{A1}, and let $\{\a_n\}$ be a sequence of positive numbers.
Let $T=F'(x^\dag)$ and for any $x\in B_\rho(x^\dag)$ let $T_x=F'(x)$. Let $0\le a\le 1/2$.
Then for $0 \le b \le 1/2+a$ there holds
\begin{align*}
&(T^*T)^a \prod_{k=j+1}^n r_{\a_k}(T^*T) \left[ g_{\a_j}(T^*T)T^*
-g_{\a_j}(T_x^*T_x)T_x^*\right] =(T^*T)^b S_j
\end{align*}
for some bounded linear operator $S_j:Y\to X$ satisfying
\begin{footnote}
{Throughout this paper we will always use $C$ to denote a generic
constant independent of $\d$ and $n$. We will also use the
convention $\Phi\lesssim \Psi$ to mean that $\Phi\le C \Psi$ for
some generic constant $C$.}
\end{footnote}
\begin{align*}
\|S_j\|\lesssim \frac{1}{\a_j} (s_n-s_{j-1})^{-1/2-a +b} K_0\|x-x^\dag\|,
\end{align*}
where $j=0, 1, \cdots, n$.
\end{lemma}

\begin{proof} Let $\eta_\a(\lambda)=(\a+\lambda)^{-1}$ and $\varphi_\a(\lambda)=g_\a(\lambda)-(\a+\lambda)^{-1}$.
We can write
$$
(T^*T)^a \prod_{k=j+1}^n r_{\a_k}(T^*T) \left[ g_{\a_j}(T^*T)T^*
-g_{\a_j}(T_x^*T_x)T_x^*\right] =J_1+J_2 +J_3,
$$
where
\begin{align*}
J_1&:=(T^*T)^a \prod_{k=j+1}^n r_{\a_k}(T^*T) g_{\a_j}(T^*T) [T^*-T_x^*],\\
J_2&:=(T^*T)^a \prod_{k=j+1}^n  r_{\a_k}(T^*T) \left[ \eta_{\a_j}(T^*T)-\eta_{\a_j}(T_x^* T_x)\right] T_x^*,\\
J_3&:=(T^*T)^a \prod_{k=j+1}^n  r_{\a_k}(T^*T) \left[ \varphi_{\a_j}(T^*T)-\varphi_{\a_j}(T_x^* T_x)\right] T_x^*.
\end{align*}
It suffices to show that for each $J_l$ there holds $J_l=(T^*T)^\nu S_l$ for some
bounded linear operator $S_l: Y\to X$ satisfying the desired estimate.
We will use the polar decomposition for linear operators which implies that $T^*=(T^*T)^{1/2} U$
for some partial isometry $U:Y\to X$.

By using Assumption \ref{A1} we have
$T^*-T_x^*=T^* (I-R_x)^*$, where $R_x:=R(x, x^\dag)$. This
together with the polar decomposition on $T^*$ gives
\begin{equation}\label{polar2}
T^*-T_x^*=(T^*T)^{1/2} U(I-R_x)^*.
\end{equation}
Consequently we can write $J_1=(T^*T)^b S_1$ with
$$
S_1=(T^*T)^{1/2+a -b} g_{\a_j}(T^*T) \prod_{k=j+1}^n r_{\a_k}(T^*T) U(I-R_x)^*.
$$
Since $0\le 1/2+a-b\le 1$, it follows from Assumption \ref{A5} that
\begin{align*}
\|S_1\| & \le \sup_{0\le \lambda\le 1} \left(\lambda^{1/2+a-b}
g_{\a_j}(\lambda) \prod_{k=j+1}^n r_{\a_k}(\lambda)\right) \|I-R_x\|\\
& \lesssim \frac{1}{\a_j} (s_n-s_{j-1})^{-1/2-a+b} K_0\|x-x^\dag\|.
\end{align*}
This shows the desired conclusion on $J_1$.

Next we consider $J_2$. Note that
\begin{align*}
\eta_{\a_j}(T^*T)-\eta_{\a_j}(T_x^*T_x)
&=(\a_j I+T^*T)^{-1} T^*(T_x-T) (\a_j I+T_x^*T_x)^{-1}\\
& +(\a_j I+T^*T)^{-1} (T_x^*-T^*) T_x(\a_j I+T_x^*T_x)^{-1}.
\end{align*}
Plugging this formula into the expression of $J_2$, and using the polar
decomposition on $T^*$ and the identity (\ref{polar2}), we have
$J_2=(T^*T)^b S_2$, where
\begin{align*}
S_2 &= \prod_{k=j+1}^n r_{\a_k}(T^*T) (\a_j I+T^*T)^{-1} (T^*T)^{1/2+a-b} U (T_x-T) (\a_j I+T_x^*T_x)^{-1} T_x^* \\
 &+ \prod_{k=j+1}^n r_{\a_k}(T^*T) (\a_j I+T^*T)^{-1} (T^*T)^{1/2+a-b} U(R_x-I)^* T_x T_x^* (\a_j I+T_x T_x^*)^{-1}.
\end{align*}
With the help of Assumption \ref{A1} we have
$$
\|(T_x-T)(\a_j I +T_x^* T_x)^{-1} T_x^*\|\le K_0\|x-x^\dag\|.
$$
Therefore, it follows from (\ref{g3}) that
\begin{align*}
\|S_2\| &\le \sup_{0\le\lambda\le 1}
\left(\lambda^{1/2+a-b} (\a_j+\lambda)^{-1} \prod_{k=j+1}^n r_{\a_k}(\lambda)\right)
\|(T_x-T)(\a_j I +T_x^*T_x)^{-1} T_x^*\|\\
& + \sup_{0\le\lambda\le 1}
\left(\lambda^{1/2+a-b} (\a_j+\lambda)^{-1} \prod_{k=j+1}^n r_{\a_k}(\lambda)\right)
\|(R_x-I)^* T_x T_x^* (\a_j I +T_x^*T_x)^{-1}\|\\
&\lesssim \a_j^{a-b-1/2} \left(1+\a_j(s_n-s_j)\right)^{-1/2-a+b} K_0\|x-x^\dag\|\\
&=\frac{1}{\a_j} (s_n-s_{j-1})^{-1/2-a+b} K_0\|x-x^\dag\|.
\end{align*}

It remains to consider $J_3$.  Since Assumption \ref{A4} implies
that $\varphi_{\a_j}(z)$ is analytic in $D_{\a_j}$, we have from
the Riesz-Dunford formula (\ref{RD}) that
\begin{align}\label{300}
J_3 =\frac{1}{2\pi i} \int_{\Gamma_{\a_j}} \varphi_{\a_j}(z) L_j(z) dz,
\end{align}
where
$$
L_j(z):=(T^*T)^a \prod_{k=j+1}^n r_{\a_k}(T^*T)
\left[(z I-T^*T)^{-1}-(z I -T_x^* T_x)^{-1}\right] T_x^*.
$$
Using the decomposition
\begin{align*}
(z I-T^*T)^{-1}-(z I-T_x^* T_x)^{-1} &=(z I -T^*T)^{-1} T^* (T-T_x) (z I-T_x^* T_x)^{-1}\\
& + (z I-T^*T)^{-1} (T^*-T_x^*) T_x (z I-T_x^* T_x)^{-1}
\end{align*}
together with the polar decomposition on $T^*$ and the identity (\ref{polar2}), we obtain
$L_j(z)=(T^*T)^b \tilde{L}_j (z)$, where
\begin{align*}
\tilde{L}_j(z)&= \prod_{k=j+1}^n r_{\a_k}(T^*T)
(z I-T^*T)^{-1}(T^*T)^{1/2+a-b} (T-T_x) (z I -T_x^* T_x)^{-1} T_x^*\\
& + \prod_{k=j+1}^n r_{\a_k}(T^*T)
(z I-T^*T)^{-1} (T^*T)^{1/2+a-b} U (I-R_x)^* T_x T_x^*(z I -T_x T_x^*)^{-1}.
\end{align*}
Combining with (\ref{300}) gives $J_3=(T^*T)^b S_3$, where
$$
S_3=\frac{1}{2\pi i} \int_{\Gamma_{\a_j}} \varphi_{\a_j}(z) \tilde{L}_j(z) dz.
$$
We need to estimate $\|S_3\|$. We first estimate $\tilde{L}_j(z)$ for $z\in \Gamma_{\a_j}$.
With the help of Assumption \ref{A1} and (\ref{200}), we have
\begin{align*}
\|(T-T_x)(z I&-T_x^*T_x)^{-1} T_x^* \|
\lesssim K_0\|x-x^\dag\|.
\end{align*}
Since $|z|\ge \a_j/2$ and $|z-\la|^{-1}\le b_0 (|z|+\la)^{-1}$ for $z\in \Gamma_{\a_j}$,
we have from (\ref{g3}) that
\begin{align*}
\|\tilde{L}_j(z)\| & \lesssim  \sup_{0\le \la \le 1}
\left(\la^{1/2+a-b} |z-\la|^{-1} \prod_{k=j+1}^n r_{\a_k}(\la)\right) K_0\|x-x^\dag\|\\
&\lesssim \sup_{0\le \la \le 1}
\left(\la^{1/2+a-b} (|z|+\lambda)^{-1} \prod_{k=j+1}^n r_{\a_k}(\la)\right) K_0\|x-x^\dag\|\\
& \lesssim |z|^{a-b-1/2} \left(1+(s_n-s_j)|z|\right)^{-1/2-a+b} K_0\|x-x^\dag\|\\
& \lesssim \a_j^{a-b-1/2} \left(1+(s_n-s_j)\a_j \right)^{-1/2-a+b} K_0\|x-x^\dag\|\\
&=\frac{1}{\a_j} (s_n-s_{j-1})^{-1/2-a+b} K_0\|x-x^\dag\|.
\end{align*}
Therefore, it follows from Assumption \ref{A4} that
\begin{align*}
\|S_3\| & \lesssim \frac{1}{\a_j} (s_n-s_{j-1})^{-1/2-a+b}
K_0\|x-x^\dag\| \int_{\Gamma_{\a_j}} |\varphi_{\a_j}(z)||d z|\\
& \lesssim \frac{1}{\a_j} (s_n-s_{j-1})^{-1/2-a+b}
K_0\|x-x^\dag\|.
\end{align*}
The proof is therefore complete.
\end{proof}

In the proof of Theorem \ref{T2} we will also need the following
inequality which can be obtained by essentially the same argument
in the proof of Lemma \ref{L10}.

\begin{lemma}\label{L11}
Let $\{g_\a\}$ satisfy Assumptions \ref{A4} and \ref{A5}, let
$F$ satisfy Assumption \ref{A1}, and let $\{\a_n\}$ be a sequence of positive numbers.
Let $T=F'(x^\dag)$ and for any $x\in B_\rho(x^\dag)$ let $T_x=F'(x)$.
Then for $0\le \mu\le 1/2$ there holds
\begin{align*}
&\left\| (T^*T)^\mu  \prod_{k=j+1}^n r_{\a_k}(T^*T) \left[ g_{\a_j}(T_x^*T_x)T_x^*
-g_{\a_j}(T_{\bar x}^*T_{\bar x})T_{\bar x}^*\right]\right\|\\
&\qquad \qquad\qquad \qquad \lesssim \frac{1}{\a_j} (s_n-s_{j-1})^{-\mu-1/2}
\left(1+K_0\|x-x^\dag\|\right) K_0\|x-\bar{x}\|
\end{align*}
for all $x, \bar{x}\in B_\rho(x^\dag)$, where $j=0, 1, \cdots, n$.
\end{lemma}

\section{\bf Rates of convergence: proof of Theorem \ref{T1}}
\setcounter{equation}{0}

We begin with the following lemma which follows from \cite[Lemma 4.3]{H2010} and its proof;
a simplified argument can be found in \cite{JT2010}.

\begin{lemma}\label{L2}
Let $\{\a_n\}$ be a sequence of positive numbers satisfying $\a_n\le c_1$, and let
$s_n$ be defined by (\ref{a1}).
Let $p\ge 0$ and $q\ge0$ be two numbers. Then we have
$$
\sum_{j=0}^n \frac{1}{\a_j} (s_n-s_{j-1})^{-p} s_j^{-q}\le C_0
s_n^{1-p-q} \left\{\begin{array}{lll}
1, & \max\{p,q\}<1,\\
\log (1+s_n), & \max\{p, q\}=1,\\
s_n^{\max\{p, q\}-1}, & \max\{p, q\}>1,
\end{array}\right.
$$
where $C_0$ is a constant depending only on $c_1$, $p$ and $q$.
\end{lemma}

In order to derive the necessary estimates on $x_n^\d-x^\dag$, we need
some useful identities. For simplicity of presentation, in this section we set
$$
e_n^\d:=x_n^\d-x^\dag, \quad  T:=F'(x^\dag)  \quad \mbox{and} \quad T_n:=F'(x_n^\d).
$$
It follows from (\ref{2}) that
\begin{align*}
e_{n+1}^\d&=e_n^\d- g_{\a_n}\left(T_n^* T_n\right) T_n^* (F(x_n^\d)-y^\d).
\end{align*}
Let
$$
u_n:=F(x_n^\d)-y- T (x_n^\d-x^\dag).
$$
Then we can write
\begin{align}\label{400}
e_{n+1}^\d&= e_n^\d- g_{\a_n}(T^*T) T^* (F(x_n^\d)-y^\d) \nonumber\\
&\quad\, - \left[g_{\a_n}(T_n^*T_n)T_n^* -g_{\a_n}(T^*T) T^*\right] (F(x_n^\d)-y^\d) \nonumber\\
&= r_{\a_n}(T^*T) e_n^\d - g_{\a_n}(T^*T) T^* (y-y^\d+ u_n) \nonumber\\
&\quad\, - \left[g_{\a_n}(T_n^*T_n)T_n^* -g_{\a_n}(T^*T)T^* \right] (F(x_n^\d)-y^\d).
\end{align}
By telescoping (\ref{400}) we can obtain
\begin{align}\label{20}
e_{n+1}^\d &= \prod_{j=0}^n r_{\a_j}(T^* T) e_0
-\sum_{j=0}^n \prod_{k=j+1}^n r_{\a_k}(T^*T) g_{\a_j}(T^*T) T^*(y-y^\d+u_j) \nonumber\\
&\quad\, - \sum_{j=0}^n \prod_{k=j+1}^n r_{\a_k}(T^*T)
\left[g_{\a_j}(T_j^*T_j)T_j^*-g_{\a_j}(T^*T)T^*\right] (F(x_j^\d)-y^\d).
\end{align}
By multiplying  (\ref{20}) by $T:=F'(x^\dag)$ and noting that
\begin{equation}\label{1111}
I-\sum_{j=0}^n \prod_{k=j+1}^n r_{\a_k}(TT^*)
g_{\a_j}(TT^*) TT^*=\prod_{j=0}^n r_{\a_j}(TT^*),
\end{equation}
we can obtain
\begin{align}\label{21}
T &e_{n+1}^\d-y^\d+y \nonumber\\
& =T \prod_{j=0}^n r_{\a_j}(T^*T) e_0
 +\prod_{j=0}^n r_{\a_j}(TT^*)(y-y^\d) - \sum_{j=0}^n \prod_{k=j+1}^n r_{\a_k}(TT^*) g_{\a_j}(TT^*) TT^*u_j \nonumber\\
&\quad\, - \sum_{j=0}^n T \prod_{k=j+1}^n
r_{\a_k}(T^*T) \left[g_{\a_j}(T_j^*T_j)T_j^*-g_{\a_j}(T^*T)T^*\right] (F(x_j^\d)-y^\d).
\end{align}

Based on (\ref{20}) and (\ref{21}) we will prove Theorem \ref{T1} concerning the order optimal
convergence rate of $x_{n_\d}^\d$ to $x^\dag$ when $e_0:=x_0-x^\dag$
satisfies the source condition (\ref{33}) for some $0<\nu\le 1/2$
and $\omega\in {\mathcal N}(F'(x^\dag))^\perp\subset X$.
We will first derive the
crucial estimates on $\|e_n^\d\|$ and $\|Te_n^\d\|$. To this end, we
introduce the integer $\tilde{n}_\d$ satisfying
\begin{equation}\label{7.15.1}
s_{\tilde{n}_\d}^{-\nu-1/2}\le \frac{(\tau-1)\d}{ 2 c_0 \|\omega\|}< s_n^{-\nu-1/2},
\qquad 0\le n<\tilde{n}_\d,
\end{equation}
where $c_0>1$ is the constant appearing in (\ref{60}).
Such $\tilde{n}_\d$ is well-defined since $s_n\rightarrow \infty$
as $n\rightarrow \infty$.

\begin{proposition}\label{P1}
Let $F$ satisfy Assumptions \ref{A1}, let $\{g_\a\}$ satisfy
Assumptions \ref{A4} and \ref{A5}, and let $\{\a_n\}$ be a sequence
of positive numbers satisfying (\ref{60}). If $x_0-x^\dag$ satisfies
(\ref{33}) for some $0<\nu \le 1/2$
and $\omega\in {\mathcal N}(F'(x^\dag))^\perp\subset X$ and if $K_0 \|\omega\|$ is
suitably small, then there exists a generic constant $C_*>0$ such
that
\begin{align}\label{30}
\|e_n^\d\| \le C_*  s_n^{-\nu} \|\omega\| \qquad \mbox{and} \qquad
\|T e_n^\d\| \le C_*  s_n^{-\nu-1/2} \|\omega\|
\end{align}
and
\begin{equation}\label{50}
\|T e_n^\d-y^\d+y\|\le (c_0 +C_* K_0\|\omega\|)
s_n^{-\nu-1/2} \|\omega\| +\d
\end{equation}
for all $0\le n\le \tilde{n}_\d$.
\end{proposition}

\begin{proof}
We will show (\ref{30}) by induction. By using (\ref{33}) and
$\|T\|\le \sqrt{\a_0}$ it is easy to see that
(\ref{30}) for $n=0$ holds if $C_*\ge 1$.
Next we assume that (\ref{30}) holds
for all $0\le n\le l$ for some $l<\tilde{n}_\d$ and show (\ref{30}) holds for $n=l+1$.

With the help of (\ref{33}) we can derive from (\ref{20}) that
\begin{align*}
\|e_{l+1}^\d\|
&\le \left\|\prod_{j=0}^l r_{\a_j}(T^*T) (T^*T)^\nu \omega\right\| +
\left\|\sum_{j=0}^l   \prod_{k=j+1}^l r_{\a_k} (T^*T)g_{\a_j}(T^*T) T^* (y-y^\d+u_j)\right\| \\
& +\left\| \sum_{j=0}^l \prod_{k=j+1}^l r_{\a_k}(T^*T)
\left[g_{\a_j}(T_j^*T_j)T_j^*-g_{\a_j}(T^*T)T^*\right]
(F(x_j^\d)-y^\d)\right\|.
\end{align*}
Thus we may use Assumption \ref{A5} and Lemma \ref{L10} with $a=b=0$ to conclude
\begin{align}\label{40}
\|e_{l+1}^\d\| &\le s_l^{-\nu}  \|\omega\| + b_2 \sum_{j=0}^l \frac{1}{\a_j}
(s_l-s_{j-1})^{-1/2}(\d+\|u_j\|) \nonumber\\
&\quad\, + C \sum_{j=0}^l \frac{1}{\a_j}
(s_l-s_{j-1})^{-1/2} K_0\|e_j\| \|F(x_j^\d)-y^\d\|.
\end{align}
Moreover, by using (\ref{33}), Assumption \ref{A5} and Lemma
\ref{L10} with $a=1/2$ and $b=0$, we have from (\ref{21}) that
\begin{align}\label{41}
\|T e_{l+1}^\d-y^\d+y\| &\le s_l^{-\nu-1/2} \|\omega\|+\d
+ b_2 \sum_{j=0}^l \frac{1}{\a_j} (s_l-s_{j-1})^{-1} \|u_j\| \nonumber\\
&\quad\, + C \sum_{j=0}^l \frac{1}{\a_j}
(s_l-s_{j-1})^{-1} K_0\|e_j\| \|F(x_j^\d)-y^\d\|.
\end{align}
With the help of Assumption \ref{A1} and the induction hypotheses, it follows for all $0\le
j\le l$ that
\begin{equation}\label{e11}
\|u_j\|\le K_0\|e_j^\d\| \|T e_j^\d\|\lesssim
K_0 \|\omega\|^2 s_j^{-2\nu -1/2}.
\end{equation}
By using the fact
\begin{equation}\label{77}
\d\le \frac{2 c_0}{\tau-1} \|\omega\| s_j^{-\nu-1/2},
\qquad 0\le j\le l
\end{equation}
and the induction hypotheses we have
\begin{equation}\label{e14}
\|F(x_j^\d)-y^\d\|\le \d +\|T e_j^\d\| +\|u_j\| \lesssim \|\omega\|
s_j^{-\nu-1/2}.
\end{equation}
In view of the estimates (\ref{e11}), (\ref{e14}), the induction hypothesis on $\|e_j\|$
and the inequality
\begin{equation}\label{7.29.1}
\sum_{j=0}^l \frac{1}{\a_j} (s_l-s_{j-1})^{-1/2}\le c_2
s_l^{1/2}
\end{equation}
for some generic constant $c_2$, which follows from Lemma \ref{L2},
we have from (\ref{40}) and (\ref{41}) that
\begin{align*}
\|e_{l+1}^\d\| & \le  \|\omega\| s_l^{-\nu}  + c_2 s_l^{1/2} \d
 +CK_0\|\omega\|^2 \sum_{j=0}^l \frac{1}{\a_j}
(s_l-s_{j-1})^{-1/2} s_j^{-2\nu-1/2}
\end{align*}
and
\begin{align*}
\|T e_{l+1}^\d-y^\d+y\| &\le \|\omega\| s_l^{-\nu-1/2} +\d  + CK_0\|\omega\|^2 \sum_{j=0}^l \frac{1}{\a_j}
(s_l-s_{j-1})^{-1} s_j^{-2\nu-1/2}.
\end{align*}
With the help of Lemma \ref{L2}, $\nu>0$, (\ref{77})
and (\ref{60}) we have
\begin{align*}
\|e_{l+1}^\d\| \le \left(1  +\frac{2}{\tau-1} c_0 c_2 +C K_0\|\omega\| \right) \|\omega\| s_l^{-\nu}
\end{align*}
and
\begin{align}\label{80}
\|Te_{l+1}^\d-y^\d+y\| &\le \d+ \left(1 +CK_0\|\omega\|\right)
\|\omega\| s_l^{-\nu-1/2} \nonumber\\
&\le \d+  \left(c_0+CK_0\|\omega\|\right) \|\omega\|
s_{l+1}^{-\nu-1/2}.
\end{align}
Consequently $\|e_{l+1}^\d\| \le C_* \|\omega\| s_{l+1}^{-\nu}$ if $C_*\ge
2+\frac{2}{\tau-1} c_0 c_2$ and $K_0\|\omega\|$ is suitably
small. Moreover, from (\ref{80}), (\ref{77}) and (\ref{60}) we also
have
\begin{align*}
\|T e_{l+1}^\d\| &\le 2\d + \left(c_0+ CK_0\|\omega\|\right)
\|\omega\| s_{l+1}^{-\nu-1/2}\\
&\le  \left(\frac{4 c_0^2}{\tau-1} +c_0 +CK_0\|\omega\|\right)
\|\omega\| s_{l+1}^{-\nu-1/2}\\
&\le C_* \|\omega\| s_{l+1}^{-\nu-1/2}
\end{align*}
if $C_*\ge 2 c_0 +\frac{4 c_0^2}{\tau-1}$ and $K_0\|\omega\|$
is suitably small. We therefore complete the proof of (\ref{30}). In
the meanwhile, (\ref{80}) gives the proof of (\ref{50}).
\end{proof}

From Proposition \ref{P1} it follows that $x_n\in B_\rho(x^\dag)$ for
$0\le n\le \tilde{n}_\d$ if $\|\omega\|$ is sufficiently small.
Furthermore, from (\ref{e11}) and (\ref{e14}) we have
\begin{equation}\label{66}
\|F(x_n^\d)-y-T e_n^\d\|\lesssim K_0\|\omega\|^2 s_n^{-2\nu-1/2}
\end{equation}
and
\begin{equation}\label{67}
\|F(x_n^\d)-y^\d\|\lesssim  \|\omega\| s_n^{-\nu-1/2}
\end{equation}
for $0\le n\le \tilde{n}_\d$.

In the following we will show that $n_\d\le \tilde{n}_\d$ for the
integer $n_\d$ defined by (\ref{DP}) with $\tau>1$.
Consequently, the method given by (\ref{2}) and (\ref{DP}) is well-defined.

\begin{lemma}\label{L44}
Let all the conditions in Proposition \ref{P1} hold. Let $\tau>1$ be a given number.
If $x_0-x^\dag$ satisfies (\ref{33}) for some $0<\nu\le 1/2$
and $\omega\in {\mathcal N}(F'(x^\dag))^\perp\subset X$ and if $K_0\|\omega\|$
is suitably small, then the discrepancy principle (\ref{DP})
defines a finite integer $n_\d$ satisfying $n_\d\le \tilde{n}_\d$.
\end{lemma}

\begin{proof}
From Proposition \ref{P1}, (\ref{66}) and $\nu> 0$
it follows for $0\le n\le \tilde{n}_\d$ that
\begin{align*}
\|F(x_n^\d)-y^\d\| & \le \|F(x_n^\d)-y-T e_n^\d\|+\|T e_n^\d -y^\d+y\|\\
&\le C K_0\|\omega\|^2 s_n^{-2\nu-1/2}
+\left(c_0 +CK_0\|\omega\|\right) s_n^{-\nu-1/2} \|\omega\|+\d\\
&\le \left(c_0 +CK_0\|\omega\|\right) s_n^{-\nu-1/2} \|\omega\| +\d.
\end{align*}
By setting $n=\tilde{n}_\d$ in the above inequality and using the definition of $\tilde{n}_\d$ we obtain
\begin{align*}
\|F(x_{\tilde{n}_\d}^\d)-y^\d\|\le \left(1+\frac{\tau-1}{2}+ CK_0\|\omega\|\right) \d\le \tau \d
\end{align*}
if $K_0\|\omega\|$ is suitably small. According to the
definition of $n_\d$ we have $n_\d\le \tilde{n}_\d$.
\end{proof}

\subsection{Completion of the proof of Theorem \ref{T1}}

From (\ref{20}), the source condition (\ref{33}), the polar decomposition on $T^*$, and Lemma \ref{L10}
with $a=0$ and $b=\nu$ it follows that
\begin{align}\label{3.47}
e_{n+1}^\d& =(T^*T)^\nu w_n,
\end{align}
where
\begin{align*}
w_n&:=\prod_{j=0}^n r_{\a_j}(T^*T) \omega -\sum_{j=0}^n S_j (F(x_j^\d)-y^\d)\\
&\quad\, -\sum_{j=0}^n \prod_{k=j+1}^n
r_{\a_k}(T^*T) g_{\a_j}(T^*T) (T^*T)^{1/2-\nu} U (y-y^\d+u_j).
\end{align*}
With the help of Assumption \ref{A4} and Lemma \ref{L10} we have
\begin{align*}
\|w_n\| &\lesssim  \|\omega\|
+ \sum_{j=0}^n \frac{1}{\a_j} (s_n-s_{j-1})^{-1/2+\nu} K_0\|e_j\|\|F(x_j^\d)-y^\d\|\\
&\quad\, + \sum_{j=0}^n \frac{1}{\a_j} (s_n-s_{j-1})^{-1/2+\nu} \left(\d+\|u_j\|\right).
\end{align*}
In view of (\ref{66}), (\ref{67}), Proposition \ref{P1}, Lemma \ref{L2} and (\ref{7.15.1})
it yields for $0\le n<\tilde{n}_\d$ that
\begin{align*}
\|w_n\| &\lesssim \|\omega\|+ \d \sum_{j=0}^n \frac{1}{\a_j} (s_n-s_{j-1})^{-1/2+\nu} \\
&\quad\, +K_0\|\omega\|^2 \sum_{j=0}^n \frac{1}{\a_j} (s_n-s_{j-1})^{-1/2+\nu} s_j^{-2\nu-1/2}\\
&\lesssim \|\omega\| + s_n^{1/2+\nu} \d\lesssim \|\omega\|.
\end{align*}
Since Lemma \ref{L44} implies that $n_\d\le \tilde{n}_\d$, we have $\|w_{n_\d-1}\|\lesssim \|\omega\|$.
On the other hand, it follows from (\ref{3.47}), Assumption \ref{A1} and the definition of $n_\d$ that
\begin{align*}
\|T (T^*T)^\nu w_{n_\d-1}\|=\|T e_{n_\d}\|\lesssim
\|F(x_{n_\d}^\d)-y\| \lesssim \d.
\end{align*}
Therefore, by using (\ref{3.47}) and the above two estimates, we
have from the interpolation inequality that
\begin{align*}
\|e_{n_\d}^\d\|&\le \|w_{n_\d-1}\|^{1/(1+2\nu)}
\|T (T^*T)^\nu w_{n_\d-1}\|^{2\nu/(1+2\nu)} \\
&\le C_\nu \|\omega\|^{1/(1+2\nu)} \d^{2\nu/(1+2\nu)}.
\end{align*}
This gives the desired estimate.

\section{\bf Convergence: proof of Theorem \ref{T2}} In this section we will show Theorem \ref{T2} concerning the convergence
of $x_{n_\d}^\d$ to $x^\dag$ as $\d\rightarrow 0$ without assuming any source conditions on $e_0:=x_0-x^\dag$.
The sequence $\{\a_n\}$ is now given by (\ref{7.19.1}). It is easy to see that
$1/\a_n\le s_n \le 1/((1-r) \a_n)$ and
\begin{equation}\label{7.19.2}
\sum_{j=0}^n \frac{1}{\a_j} (s_n-s_{j-1})^{-1} s_j^{-\mu}\le C_1 s_n^{-\mu}
\end{equation}
for $0\le \mu  <1$, where $C_1$ is a constant depending only on $r$ and $\mu$. We remark that (\ref{7.19.2})
may not be true for a general sequence
$\{\a_n\}$ satisfying (\ref{60}).

We first show that the method given by (\ref{2}) and (\ref{DP}) is well-defined. To this end, we introduce
the integer $\hat{n}_\d$ satisfying
\begin{equation}\label{7.20.5}
s_{\hat{n}_\d}^{-1/2} \le \frac{(\tau-1) \d}{2 c_0 \|e_0\|} < s_n^{-1/2}, \qquad 0\le n<\hat{n}_\d.
\end{equation}
Since $s_n\rightarrow \infty$ as $n\rightarrow \infty$, such $\hat{n}_\d$ is well-defined.

\begin{lemma}\label{L7.20.1}
Let $F$ satisfy Assumptions \ref{A1}, let $\{g_\a\}$
satisfy Assumptions \ref{A4} and \ref{A5},  and let $\{\a_n\}$ be
a sequence of positive numbers satisfying (\ref{7.19.1}). Let $\tau>1$ be a given number.
If $K_0\|e_0\|$ is suitably small, then
there is a constant $C_*$ such that
\begin{equation}\label{7.20.4}
\|e_n^\d\|\le C_* \|e_0\| \quad \mbox{and} \quad \|T e_n^\d\|\le C_* \|e_0\| s_n^{-1/2}
\end{equation}
for $0\le n\le \hat{n}_\d$, and the discrepancy principle (\ref{DP}) defines a finite integer $n_\d$ satisfying
$n_\d\le \hat{n}_\d$.
\end{lemma}

\begin{proof} We prove (\ref{7.20.4}) by induction. By using $\|T\|\le \sqrt{\a_0}$, it is easy to see
that (\ref{7.16.0}) is true for $n=0$ if $C_*\ge 1$. Next we assume that (\ref{7.16.0}) holds for
all $0\le n\le l$ for some $l<\hat{n}_\d$, and show that
it is also true for $n=l+1$. By a similar argument in the proof of Proposition \ref{P1} we obtain
\begin{align}
\|e_{l+1}^\d\| &\le \|e_0\|+c_2 s_l^{1/2} \d +CK_0\|e_0\|^2 \sum_{j=0}^l \frac{1}{\a_j} (s_l-s_{j-1})^{-1/2} s_j^{-1/2} \label{7.19.3}
\end{align}
and
\begin{align}
\|T e_{l+1}^\d-y^\d+y\| &\le \|e_0\| s_l^{-1/2}+\d +CK_0\|e_0\|^2 \sum_{j=0}^l \frac{1}{\a_j} (s_l-s_{j-1})^{-1} s_j^{-1/2}.\label{7.19.4}
\end{align}
By using (\ref{7.20.5}) and Lemma \ref{L2} we obtain from (\ref{7.19.3}) that
$$
\|e_{l+1}^\d\|\le \left(1+ \frac{2}{\tau-1} c_0c_2 +CK_0\|e_0\|\right) \|e_0\| \le C_* \|e_0\|
$$
if $C_*\ge 2+\frac{2}{\tau-2} c_0c_2$ and $K_0\|e_0\|$ is suitably small.
On the other hand, by using (\ref{7.19.2}) with $\mu=1/2$ and (\ref{60}) we obtain from
(\ref{7.19.4}) that
\begin{align}
\|T e_{l+1}^\d-y^\d+y\| &\le \d+ (1+CK_0\|e_0\|) \|e_0\| s_l^{-1/2} \nonumber\\
&\le \d + (c_0+ CK_0\|e_0\|) \|e_0\| s_{l+1}^{-1/2}.
\end{align}
Consequently, we have from (\ref{7.20.5}) that
$$
\|T e_{l+1}^\d\|\le \left(\frac{4 c_0^2}{\tau-1} +c_0 +CK_0\|e_0\|\right) \|e_0\| s_{l+1}^{-1/2} \le C_* \|e_0\| s_{l+1}^{-1/2}
$$
if $C_*\ge 2 c_0 +\frac{4 c_0^2}{\tau-1}$ and $K_0\|e_0\|$ is suitably small. We thus complete the proof of (\ref{7.20.4}).

Note that the above argument in fact shows also that
$$
\|T e_n^\d-y^\d+y\|\le \d +(c_0+CK_0\|e_0\|) \|e_0\| s_n^{-1/2}, \qquad 0\le n\le \hat{n}_\d.
$$
Thus, by the similar argument in the proof of Lemma \ref{L44} we can derive $\|F(x_{\hat{n}_\d}^\d)-y^\d\|\le \tau \d$
if $K_0\|e_0\|$ is suitably small. According to the definition of $n_\d$ we obtain $n_\d\le \hat{n}_\d$.
\end{proof}

In the remaining part of this section we will show $x_{n_\d}^\d\rightarrow x^\dag$ as
$\d\rightarrow 0$. We will achieve this by first considering the noise free iterative
sequence $\{x_n\}$ defined by (\ref{2}) with $y^\d$ replaced by $y$, i.e.
\begin{equation}\label{7.15.2}
x_{n+1}=x_n-g_{\a_n}\left(F'(x_n)^*F'(x_n)\right) F'(x_n)^* (F(x_n)-y)
\end{equation}
and showing that $x_n\rightarrow x^\dag$ as $n\rightarrow \infty$. We then derive the
stability estimate on $\|x_n^\d-x_n\|$ for $0\le n\le n_\d$ together with other related estimates.
With the help of the definition of $n_\d$, we will be able to show the convergence of $x_{n_\d}^\d$ to $x^\dag$
as $\d\rightarrow 0$.

\subsection{Convergence of the noise free iteration}

In this subsection we will show the convergence of $x_n$ to $x^\dag$ as $n\rightarrow \infty$.
We first show that if $x_0-x^\dag\in {\mathcal R}(F'(x^\dag)^*)$
then $x_n\rightarrow x^\dag$ as $n\rightarrow \infty$. We then perturb the initial guess $x_0$ to be $\hat{x}_0$
such that $\hat{x}_0-x^\dag\in {\mathcal R}(F'(x^\dag)^*)$ and define $\{\hat{x}_n\}$ by
\begin{equation}\label{7.15.3}
\hat{x}_{n+1}=\hat{x}_n -g_{\a_n}\left(F'(\hat{x}_n)^* F'(\hat{x}_n)\right) F'(\hat{x}_n)^* (F(\hat{x}_n)-y).
\end{equation}
Since $x_0-x^\dag\in {\mathcal N}(F'(x^\dag))^\perp=\overline{{\mathcal R}(F'(x^\dag)^*)}$, such $\hat{x}_0$ can be chosen as
close to $x_0$ as we want. We then show that $\{x_n\}$ is stable relative to the change of $x_0$.
This allows us to derive the convergence of $\{x_n\}$.

We start with several lemmas. We first show that $x_n$ is
well-defined for all $n$ and satisfies certain estimates.

\begin{lemma}\label{L7.16.0}
Let $F$ satisfy Assumptions \ref{A1}, let $\{g_\a\}$
satisfy Assumptions \ref{A4} and \ref{A5},  and let $\{\a_n\}$ be
a sequence of positive numbers satisfying (\ref{7.19.1}). If $K_0\|e_0\|$ is suitably small, then
\begin{equation}\label{7.16.0}
\|e_n\|\le 2\|e_0\| \quad \mbox{and} \quad \|T e_n\|\le 2c_0 \|e_0\| s_n^{-1/2}
\end{equation}
for $n=0,1,\cdots$, where $e_n:=x_n-x^\dag$.
\end{lemma}

\begin{proof} This result can be obtained by the same argument in the proof of Lemma \ref{L7.20.1}.
\end{proof}

\begin{lemma}\label{L7.19.2}
Let $F$ satisfy Assumptions \ref{A1}, let $\{g_\a\}$
satisfy Assumptions \ref{A4} and \ref{A5},  and let $\{\a_n\}$ be
a sequence of positive numbers satisfying (\ref{7.19.1}).
If $e_0=(T^*T)^{1/4} \omega$
for some $\omega\in {\mathcal N}(T)^\perp\subset X$ and $K_0\|e_0\|$ is suitably small, then
\begin{equation}\label{7.19.6}
\|e_n\|\le 2c_0 \|\omega\| s_n^{-1/4} \quad \mbox{and}
\quad \|T e_n\|\le 2 c_0 \|\omega\| s_n^{-3/4}
\end{equation}
for $n=0,1,\cdots$.
\end{lemma}

\begin{proof}
We prove (\ref{7.19.6}) by induction. By using $\|T\|\le \sqrt{\a_0}$ and
$e_0=(T^*T)^{1/4} \omega$, it is easy to see that (\ref{7.19.6}) is true for $n=0$. Next
we assume that (\ref{7.19.6}) holds for all $0\le n\le l$, and show that it also holds for $n=l+1$.
By a similar argument in the proof of Proposition \ref{P1} we obtain
\begin{align}\label{7.20.1}
\|e_{l+1}\| &\le s_l^{-1/4} \|\omega\| + b_2 \sum_{j=0}^l \frac{1}{\a_j} (s_l-s_{j-1})^{-1/2}
\|F(x_j)-y-T e_j\| \nonumber\\
&\quad \, + C \sum_{j=0}^l \frac{1}{\a_j} (s_l-s_{j-1})^{-1/2} K_0\|e_j\| \|F(x_j)-y\|
\end{align}
and
\begin{align}\label{7.20.2}
\|T e_{l+1}\| &\le s_l^{-3/4} \|\omega\| + b_2 \sum_{j=0}^l \frac{1}{\a_j} (s_l-s_{j-1})^{-1}
\|F(x_j)-y-T e_j\| \nonumber\\
&\quad \, + C \sum_{j=0}^l \frac{1}{\a_j} (s_l-s_{j-1})^{-1} K_0\|e_j\| \|F(x_j)-y\|.
\end{align}
With the help of Assumption \ref{A1}, Lemma \ref{L7.16.0} and the induction hypotheses, we have
for $0\le j\le l$ that
\begin{align*}
&\|F(x_j)-y-T e_j\|\le K_0\|e_j\|\|T e_j\|\le CK_0\|e_0\| \|\omega\| s_j^{-3/4},\\
&\|F(x_j)-y\|\le \|T e_j\| +\|F(x_j)-y -T e_j\|\lesssim  \|\omega\| s_j^{-3/4}.
\end{align*}
Therefore, by using Lemma \ref{L2}, we obtain from (\ref{7.20.1}) that
\begin{align*}
\|e_{l+1}\| &\le s_l^{-1/4} \|\omega\| +C K_0\|e_0\| \|\omega\| \sum_{j=0}^l \frac{1}{\a_j} (s_l-s_{j-1})^{-1/2} s_j^{-3/4}\\
&\le \left(1+CK_0\|e_0\|\right) \|\omega\| s_l^{-1/4},
\end{align*}
while by using (\ref{7.19.2}) with $\mu=3/4$ we obtain
\begin{align*}
\|T e_{l+1}\| &\le s_l^{-3/4} \|\omega\| + C K_0\|e_0\| \|\omega\| \sum_{j=0}^l \frac{1}{\a_j} (s_l-s_{j-1})^{-1} s_j^{-3/4}\\
&\le \left(1+ C K_0\|e_0\|\right) \|\omega\| s_l^{-3/4}.
\end{align*}
Thus, by using $s_{l+1}\le c_0 s_l$, we obtain for suitably small $K_0\|e_0\|$ that $\|e_{l+1}\|\le 2 c_0 \|\omega\| s_{l+1}^{-1/4}$
and $\|T e_{l+1}\|\le 2 c_0 \|\omega\| s_{l+1}^{-3/4}$. The proof is therefore complete.
\end{proof}

We remark that the crucial point in Lemma \ref{7.19.2} is that it
requires only the smallness of $K_0\|e_0\|$, which is different from
proposition \ref{P1} where the smallness of $K_0\|\omega\|$ is needed.
This will allow us to pass through the approximation argument due to the perturbation of the initial guess.

We now derive a perturbation result on $\|x_n-\hat{x}_n\|$ and $\|T(x_n-\hat{x}_n)\|$
relative to the change of the initial guess. For simplicity of the presentation we set
$$
\hat{e}_n:=\hat{x}_n-x^\dag, \qquad T_n=F'(x_n), \qquad \hat{T}_n =F'(\hat{x}_n).
$$
It follows from (\ref{7.15.2}) and (\ref{7.15.3}) that
\begin{align*}
x_{n+1}-\hat{x}_{n+1} &=x_n-\hat{x}_n -g_{\a_n}(T_n^* T_n) T_n^* (F(x_n)-y)
+ g_{\a_n}(\hat{T}_n^* \hat{T}_n) \hat{T}_n^* (F(\hat{x}_n)-y)\\
&= r_{\a_n}(T^*T) (x_n-\hat{x}_n)-g_{\a_n}(T^*T)T^*\left(F(x_n)-F(\hat{x}_n)-T (x_n-\hat{x}_n)\right)\\
&\quad\, -\left[g_{\a_n}(T_n^*T_n)T_n^*-g_{\a_n}(T^*T)T^*\right] \left(F(x_n)-F(\hat{x}_n)\right)\\
&\quad \, -\left[g_{\a_n}(T_n^*T_n)T_n^*-g_{\a_n}(\hat{T}_n^*\hat{T}_n)\hat{T}_n^* \right] \left(F(\hat{x}_n)-y\right).
\end{align*}
By telescoping this identity we obtain
\begin{align}\label{7.15.4}
x_{n+1}-\hat{x}_{n+1} &=\prod_{k=0}^n r_{\a_k}(T^*T) (x_0-\hat{x}_0) \nonumber\\
& -\sum_{j=0}^n \prod_{k=j+1}^n r_{\a_k}(T^*T) g_{\a_j}(T^*T)T^*\left(F(x_j)-F(\hat{x}_j)-T (x_j-\hat{x}_j)\right) \nonumber\\
& -\sum_{j=0}^n \prod_{k=j+1}^n r_{\a_k}(T^*T)\left[g_{\a_j}(T_j^*T_j)T_j^*-g_{\a_j}(T^*T)T^*\right] \left(F(x_j)-F(\hat{x}_j)\right) \nonumber\\
& -\sum_{j=0}^n \prod_{k=j+1}^n r_{\a_k}(T^*T)\left[g_{\a_j}(T_j^*T_j)T_j^*-g_{\a_j}(\hat{T}_j^*\hat{T}_j)\hat{T}_j^* \right] \left(F(\hat{x}_j)-y\right).
\end{align}

\begin{lemma}\label{L9}
Let $F$ satisfy Assumptions \ref{A1}, let $\{g_\a\}$
satisfy Assumptions \ref{A4} and \ref{A5},  and let $\{\a_n\}$ be
a sequence of positive numbers satisfying (\ref{7.19.1}). If $K_0\|e_0\|$ and $K_0\|\hat{e}_0\|$
are suitably small, then
\begin{equation}\label{7.15.5}
\|x_n-\hat{x}_n\|\le 2\|x_0-\hat{x}_0\| \quad \mbox{and} \quad \|T(x_n-\hat{x}_n)\|\le 2 c_0 s_n^{-1/2} \|x_0-\hat{x}_0\|
\end{equation}
for $n=0,1,\cdots$.
\end{lemma}

\begin{proof} We will show (\ref{7.15.5}) by induction. Since $\|T\|\le \sqrt{\a_0}$, (\ref{7.15.5}) holds for $n=0$.
In the following we will assume that (\ref{7.15.5}) holds for $0\le n\le l$, and show that it is also true for $n=l+1$.

In view of Assumption \ref{A5}, Lemma \ref{L10} with $a=b=0$, and Lemma \ref{L11} with $\mu=0$,
it follows from (\ref{7.15.4}) that
\begin{align}\label{7.15.6}
\|x_{l+1}-\hat{x}_{l+1}\| &\le \|x_0-\hat{x}_0\| \nonumber\\
& + C \sum_{j=0}^l \frac{1}{\a_j} (s_l-s_{j-1})^{-1/2} \|F(x_j)-F(\hat{x}_j)-T (x_j-\hat{x}_j)\| \nonumber\\
& + C \sum_{j=0}^l \frac{1}{\a_j} (s_l-s_{j-1})^{-1/2} K_0\|e_j\| \|F(x_j)-F(\hat{x}_j)\| \nonumber\\
& + C \sum_{j=0}^l \frac{1}{\a_j} (s_l-s_{j-1})^{-1/2} K_0\|x_j-\hat{x}_j\|\|F(\hat{x}_j)-y\|.
\end{align}
Next we multiply (\ref{7.15.4}) by $T$.  By using Assumption \ref{A5}, Lemma \ref{L10} with $a=1/2$ and $b=0$,
and Lemma \ref{L11} with $\mu=1/2$, we obtain
\begin{align}\label{7.15.7}
\|T(x_{l+1}-\hat{x}_{l+1})\| &\le s_l^{-1/2} \|x_0-\hat{x}_0\| \nonumber\\
& + C \sum_{j=0}^l \frac{1}{\a_j} (s_l-s_{j-1})^{-1} \|F(x_j)-F(\hat{x}_j)-T (x_j-\hat{x}_j)\| \nonumber\\
& + C \sum_{j=0}^l \frac{1}{\a_j} (s_l-s_{j-1})^{-1} K_0\|e_j\| \|F(x_j)-F(\hat{x}_j)\| \nonumber\\
& + C \sum_{j=0}^l \frac{1}{\a_j} (s_l-s_{j-1})^{-1} K_0\|x_j-\hat{x}_j\|\|F(\hat{x}_j)-y\|.
\end{align}
From Lemma \ref{L7.16.0} and Assumption \ref{A1} it follows that
\begin{equation}\label{7.16.1}
\|F(\hat{x}_j)-y\|\lesssim \|T \hat{e}_j\|+K_0\|\hat{e}_j\|\|T \hat{e}_j\| \lesssim s_j^{-1/2} \|\hat{e}_0\|.
\end{equation}
Moreover, By using Assumption \ref{A1} we have
\begin{align*}
F(x_j)-F(\hat{x}_j)-T(x_j-\hat{x}_j) &=\int_0^1 \left[F'(\hat{x}_j+t(x_j-\hat{x}_j))-T\right] (x_j-\hat{x}_j) dt\\
&=\int_0^1 \left[R(\hat{x}_j+t(x_j-\hat{x}_j), x^\dag)-I\right] T(x_j-\hat{x}_j) dt.
\end{align*}
Consequently
\begin{align*}
\|F(x_j)-F(\hat{x}_j)-T(x_j-\hat{x}_j)\|
&\le \int_0^1 \|R(\hat{x}_j+t(x_j-\hat{x}_j), x^\dag)-I\| \|T(x_j-\hat{x}_j)\| dt\\
&\le \frac{1}{2} K_0\left(\|e_j\|+\|\hat{e}_j\|\right) \|T(x_j-\hat{x}_j)\|.
\end{align*}
With the help of Lemma \ref{L7.16.0} it yields
\begin{align}\label{7.16.2}
\|F(x_j)-F(\hat{x}_j)-T(x_j-\hat{x}_j)\|\le K_0\left(\|e_0\|+\|\hat{e}_0\|\right) \|T(x_j-\hat{x}_j)\|.
\end{align}
This in particular implies
\begin{equation}\label{7.16.3}
\|F(x_j)-F(\hat{x}_j)\|\le 2\|T(x_j-\hat{x}_j)\|.
\end{equation}
By virtue of (\ref{7.16.1}), (\ref{7.16.2}), (\ref{7.16.3}) and the induction hypotheses, we have from
(\ref{7.15.6}) and (\ref{7.15.7}) that
\begin{align}\label{7.16.4}
\|x_{l+1}-\hat{x}_{l+1}\| &\le \|x_0-\hat{x}_0\| \nonumber\\
& +C K_0 (\|e_0\|+\|\hat{e}_0\|) \|x_0-\hat{x}_0\| \sum_{j=0}^l \frac{1}{\a_j} (s_l-s_{j-1})^{-1/2} s_j^{-1/2}
\end{align}
and
\begin{align}\label{7.16.5}
\|T(x_{l+1}-\hat{x}_{l+1})\| &\le s_l^{-1/2} \|x_0-\hat{x}_0\| \nonumber\\
& +CK_0(\|e_0\|+\|\hat{e}_0\|) \|x_0-\hat{x}_0\| \sum_{j=0}^l \frac{1}{\a_j} (s_l-s_{j-1})^{-1} s_j^{-1/2}.
\end{align}
With the help of Lemma \ref{L2} and (\ref{7.19.2}) we can derive
\begin{align*}
\|x_{l+1}-\hat{x}_{l+1}\| &\le \left(1+CK_0(\|e_0\|+\|\hat{e}_0\|)\right) \|x_0-\hat{x}_0\| \\
&\le 2\|x_0-\hat{x}_0\|
\end{align*}
and
\begin{align*}
\|T(x_{l+1}-\hat{x}_{l+1})\| &\le \left(1+CK_0(\|e_0\|+\|\hat{e}_0\|)\right)  s_l^{-1/2} \|x_0-\hat{x}_0\|\\
&\le 2 c_0 s_{l+1}^{-1/2} \|x_0-\hat{x}_0\|
\end{align*}
if $K_0\|e_0\|$ and $K_0\|\hat{e}_0\|$ are suitably small. The proof is thus complete.
\end{proof}

\begin{theorem}\label{T3}
Let $F$ satisfy Assumptions \ref{A1}, let $\{g_\a\}$
satisfy Assumptions \ref{A4} and \ref{A5},  and let $\{\a_n\}$ be
a sequence of positive numbers satisfying (\ref{7.19.1}). If $e_0\in {\mathcal N}(T)^\perp$
and $K_0\|e_0\|$ is suitably small, then
\begin{equation}\label{5.6}
\lim_{n\rightarrow \infty} \|x_n-x^\dag\|=0 \quad \mbox{and} \quad
\lim_{n\rightarrow \infty} s_n^{1/2} \|T(x_n-x^\dag)\|=0
\end{equation}
for the sequence $\{x_n\}$ defined by (\ref{7.15.2}).
\end{theorem}

\begin{proof}
Let $0<\varepsilon<\|e_0\|$ be an arbitrarily small number. Since
$e_0\in {\mathcal N}(T)^\perp =\overline{{\mathcal
R}(T^*)}$, there is an $\hat{x}_0\in X$ such that
$\hat{e}_0:=\hat{x}_0-x^\dag\in {\mathcal R}(T^*)$ and
$\|x_0-\hat{x}_0\|<\varepsilon$.  Note that
$K_0\|\hat{e}_0\|\le 2 K_0\|e_0\|$. Thus, if
$K_0\|e_0\|$ is suitably small, then for the sequence
$\{\hat{x}_n\}$ defined by (\ref{7.15.3}), it follows from Lemma
\ref{L9} that
$$
\|x_n-\hat{x}_n\|\le 2\|x_0-\hat{x}_0\|<2\varepsilon
$$
and
$$
s_n^{1/2} \|T(x_n-\hat{x}_n)\|\le
2c_0\|x_0-\hat{x}_0\|< 2c_0\varepsilon
$$
for all $n\ge 0$. On the other hand, since $\hat{e}_0\in {\mathcal R}(T^*)={\mathcal R}((T^*T)^{1/2})
\subset {\mathcal R}((T^*T)^{1/4})$, from Lemma \ref{L7.19.2} we have
$\|\hat{e}_n\|\rightarrow 0$ and $s_n^{1/2} \|T \hat{e}_n\| \rightarrow 0$
as $n\rightarrow \infty$. Thus, there
is a $n_0$ such that $\|\hat{e}_n\|<\varepsilon$ and $s_n^{1/2}
\|T \hat{e}_n\|< c_0 \varepsilon$ for all $n\ge n_0$.
Consequently
$$
\|e_n\|\le \|x_n-\hat{x}_n\|+\|\hat{e}_n\|<3\varepsilon
$$
and
$$
s_n^{1/2}\| T e_n\|\le s_n^{1/2} \|T (x_n-\hat{x}_n)\|
+s_n^{1/2} \|T \hat{e}_n\|<3c_0
\varepsilon
$$
for all $n\ge n_0$. Since $\varepsilon>0$ is arbitrarily small, we
therefore obtain (\ref{5.6}).
\end{proof}

\subsection{Stability estimates} In this subsection we will derive the stability estimates
on $\|x_n^\d-x_n\|$ for $0\le n \le \hat{n}_\d$, where
$\hat{n}_\d$ is defined by (\ref{7.20.5}). We will use the notations
$$
\A:=F'(x^\dag)^* F'(x^\dag), \quad \A_n:=F'(x_n)^* F'(x_n), \quad \A_n^\d:=F'(x_n^\d)^* F'(x_n^\d).
$$
The main result is as follows.

\begin{proposition}\label{P10}
Let $F$ satisfy Assumptions \ref{A1}, let $\{g_\a\}$
satisfy Assumptions \ref{A4} and \ref{A5},  and let $\{\a_n\}$ be
a sequence of positive numbers satisfying (\ref{7.19.1}). If $K_0\|e_0\|$ is suitably small, then
\begin{equation}\label{7.16.6}
\|x_n^\d-x_n\|\lesssim s_n^{1/2} \d
\end{equation}
and
\begin{equation}\label{7.16.7}
\|F(x_n^\d)-F(x_n)-y^\d+y\|\le \left(1+CK_0\|e_0\|\right) \d
\end{equation}
for $0\le n\le \hat{n}_\d$.
\end{proposition}

\begin{proof}
We first show (\ref{7.16.6}) by establishing
\begin{equation}\label{stable}
\|x_n^\d-x_n\|\le 2 b_2 c_2 s_n^{1/2} \d \qquad \mbox{and} \qquad \|T(x_n^\d-x_n)\|\le 3 \d
\end{equation}
for $0\le n\le \hat{n}_\d$, where $b_2$ and $c_2$ are the constants appearing in (\ref{g2}) and
(\ref{7.29.1}) respectively. It is clear that (\ref{stable}) is true for $n=0$.
Now we assume that (\ref{stable}) is true for all
$0\le n\le l$ for some $l<\hat{n}_\d$ and show that it is also true for $n=l+1$. We set
\begin{align*}
v_n&:=F(x_n^\d)-F(x_n)-T(x_n^\d-x_n),\\
w_n&:=F(x_n^\d)-F(x_n)-y^\d+y.
\end{align*}
It then follows from the definition of $\{x_n^\d\}$ and $\{x_n\}$ that
\begin{align*}
x_{n+1}^\d-x_{n+1}&=x_n^\d-x_n-g_{\a_n}(\A_n^\d) F'(x_n^\d)^* (F(x_n^\d)-y^\d)\\
& \quad\, +g_{\a_n}(\A_n)F'(x_n)^* (F(x_n)-y)\\
&= r_{\a_n}(\A) (x_n^\d-x_n)- g_{\a_n}(\A)F'(x^\dag)^* \left(v_n-y^\d+y\right)\\
&\quad\, -\left[g_{\a_n}(\A_n) F'(x_n)^*-g_{\a_n}(\A) F'(x^\dag)^*\right] w_n \\
&\quad \, -\left[g_{\a_n}(\A_n^\d) F'(x_n^\d)^*- g_{\a_n}(\A_n) F'(x_n)^*\right] \left(F(x_n^\d)-y^\d\right).
\end{align*}
By telescoping the above equation and noting that $x_0^\d=x_0$ we obtain
\begin{align}\label{7.16.8}
x_{l+1}^\d-x_{l+1}
&= \sum_{j=0}^l \prod_{k=j+1}^l r_{\a_k}(\A) g_{\a_j}(\A)F'(x^\dag)^* \left(y^\d-y- v_j\right) \nonumber\\
& -\sum_{j=0}^l \prod_{k=j+1}^l r_{\a_k}(\A)
\left[g_{\a_j}(\A_j^\d) F'(x_j^\d)^*- g_{\a_j}(\A_j) F'(x_j)^*\right] \left(F(x_j^\d)-y^\d\right)\nonumber\\
& -\sum_{j=0}^l \prod_{k=j+1}^l r_{\a_k}(\A) \left[g_{\a_j}(\A_j) F'(x_j)^*-g_{\a_j}(\A) F'(x^\dag)^*\right] w_j.
\end{align}
In view of Assumption \ref{A4}, Lemma \ref{L10} and Lemma \ref{L11} it follows that
\begin{align}\label{7.16.9}
\|x_{l+1}^\d-x_{l+1}\|
&\le  b_2 \sum_{j=0}^l \frac{1}{\a_j} (s_l-s_{j-1})^{-1/2}  \left(\d+\|v_j\|\right) \nonumber\\
& + C \sum_{j=0}^l \frac{1}{\a_j} (s_l-s_{j-1})^{-1/2} K_0\|x_j^\d-x_j\| \left\|F(x_j^\d)-y^\d\right\|\nonumber\\
& + C \sum_{j=0}^l \frac{1}{\a_j} (s_l-s_{j-1})^{-1/2} K_0\|e_j\| \|w_j\|.
\end{align}
By multiplying (\ref{7.16.8}) by $T$ and using  (\ref{1111}) we obtain with $\B=F'(x^\dag) F'(x^\dag)$ that
\begin{align*}
T (&x_{l+1}^\d-x_{l+1})-y^\d+y\nonumber\\
&= \prod_{j=0}^l r_{\a_j}(\B) (y-y^\d) - \sum_{j=0}^l \prod_{k=j+1}^l r_{\a_k}(\B) g_{\a_j}(\B)\B v_j \nonumber\\
& -\sum_{j=0}^l T\prod_{k=j+1}^l r_{\a_k}(\A)
\left[g_{\a_j}(\A_j^\d) F'(x_j^\d)^*- g_{\a_j}(\A_j) F'(x_j)^*\right] \left(F(x_j^\d)-y^\d\right)\nonumber\\
& -\sum_{j=0}^l T\prod_{k=j+1}^l r_{\a_k}(\A) \left[g_{\a_j}(\A_j) F'(x_j)^*-g_{\a_j}(\A) F'(x^\dag)^*\right] w_j.
\end{align*}
It follows from Assumption \ref{A4}, Lemma \ref{L10} and Lemma \ref{L11} that
\begin{align}\label{7.16.10}
\|T &(x_{l+1}^\d-x_{l+1})-y^\d+y\| \nonumber\\
&\le \d + C \sum_{j=0}^l \frac{1}{\a_j} (s_l-s_{j-1})^{-1} \|v_j\|
+ C\sum_{j=0}^l \frac{1}{\a_j} (s_l-s_{j-1})^{-1} K_0\|e_j\| \| w_j\|\nonumber\\
&\quad \, + C \sum_{j=0}^l \frac{1}{\a_j}(s_l-s_{j-1})^{-1} K_0\|x_j^\d-x_j\| \|F(x_j^\d)-y^\d\|.
\end{align}
With the help of Assumption \ref{A1}, Lemma \ref{L7.20.1}, Lemma \ref{L7.16.0} and the induction hypotheses
we have
$$
\|v_j\|\lesssim K_0\left(\|e_j\|+\|e_j^\d\|\right) \|T(x_j^\d-x_j)\|\lesssim K_0\|e_0\| \d
$$
and
$$
\|w_j\|\le \d +\|T(x_j^\d-x_j)\|+\|v_j\|\lesssim \d.
$$
Moreover, by using Assumption \ref{A1}, Lemma \ref{L7.20.1} and (\ref{7.20.5}) we have for $0\le j\le l$
$$
\|F(x_j^\d)-y^\d\|\le \d+\|T e_j^\d\|+ K_0\|e_j^\d\| \|T e_j^\d\|\lesssim \d +s_j^{-1/2} \|e_0\|\lesssim s_j^{-1/2} \|e_0\|.
$$
Combining the above three inequalities with (\ref{7.16.9}) and (\ref{7.16.10}) and using the induction hypothesis
$\|x_j^\d-x_j\|\lesssim s_j^{1/2} \d$ for $0\le j\le l$ it follows that
\begin{align}\label{7.16.11}
\|x_{l+1}^\d-x_{l+1}\|
&\le  (b_2+CK_0\|e_0\|) \d \sum_{j=0}^l \frac{1}{\a_j} (s_l-s_{j-1})^{-1/2}
\end{align}
and
\begin{align}\label{7.16.12}
\|T (x_{l+1}^\d-x_{l+1})-y^\d+y\|\le  \d + K_0\|e_0\| \d \sum_{j=0}^l \frac{1}{\a_j} (s_l-s_{j-1})^{-1}.
\end{align}
By Lemma \ref{L2} and the fact $s_l\le s_{l+1}$ we obtain for small $K_0\|e_0\|$ that
$$
\|x_{l+1}^\d-x_{l+1}\|\le (b_2 c_2 +CK_0\|e_0\|) s_l^{1/2} \d\le 2 b_2 c_2 s_{l+1}^{1/2} \d
$$
Moreover, with the help of (\ref{7.19.2}) we can derive
\begin{align*}
\|T (x_{l+1}^\d-x_{l+1})-y^\d+y\|\le  (1+C K_0\|e_0\|) \d.
\end{align*}
Thus $\|T(x_{l+1}^\d-x_{l+1})\|\le 3 \d$ if $K_0\|e_0\|$ is suitably small.
We therefore complete the proof of (\ref{stable}).

Next we will prove (\ref{7.16.7}). From the above proof we in fact obtain
$$
\|T(x_n^\d-x_n)-y^\d+y\|\le (1+CK_0\|e_0\|) \d, \qquad 0\le n\le \hat{n}_\d.
$$
Therefore, it follows from Assumption \ref{A1}, Lemma \ref{L7.20.1} and Lemma \ref{L7.16.0} that
\begin{align*}
\|F(x_n^\d)&-F(x_n)-y^\d+y\|\\
&\le \|F(x_n^\d)-F(x_n)-T(x_n^\d-x_n)\|+ \|T(x_n^\d-x_n)-y^\d+y\|\\
&\le K_0(\|e_n^\d\|+\|e_n\|) \|T(x_n^\d-x_n)\| +(1+CK_0\|e_0\|) \d\\
&\le (1+CK_0\|e_0\|) \d.
\end{align*}
The proof is thus complete.
\end{proof}

\subsection{Completion of the proof of Theorem \ref{T2}}

We have shown in Lemma \ref{L7.20.1} that $n_\d\le \hat{n}_\d$. Thus we may use
the definition of $n_\d$ and Proposition \ref{P10} to obtain
\begin{equation}\label{8.1}
\|F(x_{n_\d})-y\|\le \|F(x_{n_\d}^\d)-y^\d\|
+\|F(x_{n_\d}^\d)-F(x_{n_\d})-y^\d+y\|\lesssim \d
\end{equation}
and for $0\le n<n_\d$
\begin{align*}
\tau \d &\le \|F(x_n^\d)-y^\d\| \le \|F(x_n^\d)-F(x_n)-y^\d+y\|+\|F(x_n)-y\|\\
&\le \left(1+CK_0\|e_0\|\right) \d +\|F(x_n)-y\|.
\end{align*}
Since $\tau>1$, if $K_0\|e_0\|$ is suitably small  then
\begin{equation}\label{8.2}
\d\lesssim \|F(x_n)-y\|\lesssim \|T e_n\|, \qquad 0\le
n<n_\d.
\end{equation}

We now prove the convergence of $x_{n_\d}^\d$ to $x^\dag$ as $\d\rightarrow 0$.
Assume first that there is a sequence
$\d_k\searrow 0$ such that $n_k:=n_{\d_k}\rightarrow n$ as
$k\rightarrow \infty$ for some finite integer $n$. Without loss of
generality, we can assume that $n_k=n$ for all $k$. It then
follows from (\ref{8.1}) that $F(x_n)=y$.  Thus, from (\ref{7.15.2})
we can conclude that $x_j=x_n$ for all $j\ge n$. Since Theorem \ref{T3}
implies $x_j\rightarrow x^\dag$ as $j\rightarrow \infty$,
we must have $x_n=x^\dag$, which together with Proposition
\ref{P10} implies $x_{n_k}^{\d_k}\rightarrow x^\dag$ as
$k\rightarrow \infty$.

Assume next that there is a sequence $\d_k\searrow 0$ such that
$n_k:=n_{\d_k}\rightarrow \infty$ as $k\rightarrow \infty$. Then
Theorem \ref{T3} and (\ref{8.2}) imply that
$\|e_{n_k}\|\rightarrow 0$ and $s_{n_k}^{1/2} \d_k \rightarrow 0$
as $k\rightarrow \infty$. Consequently, by Proposition \ref{P10} we
again obtain $x_{n_k}^{\d_k}\rightarrow x^\dag$ as $k\rightarrow
\infty$.

\section{\bf Numerical results}

In this section we present some numerical results to test the
theoretical conclusions given in Theorems \ref{T1} and \ref{T2}.
We consider the estimation of the
coefficient $c$ in the two-point boundary value problem
\begin{equation}\label{6.1}
\left\{\begin{array}{lll}
-u''+c u=f&\quad  \mbox{in } (0,1)\\
u(0)=g_0, & u(1)=g_1&
\end{array}\right.
\end{equation}
from the $L^2$ measurement $u^\delta$ of the state variable $u$,
where $g_0$, $g_1$ and $f\in L^2[0,1]$ are given. This inverse
problem reduces to solving (\ref{1}) with the
nonlinear operator $F: D(F)\subset L^2[0,1]\mapsto L^2[0,1]$
defined as the parameter-to-solution mapping $F(c):=u(c)$, where
$u(c)$ denotes the unique solution of (\ref{6.1}). It is well
known that $F$ is well-defined on
$$
D(F):=\left\{c\in
 L^2[0,1]:\|c-\hat{c}\|_{L^2}\le\gamma \mbox{ for some }
\hat{c}\ge 0~\hbox{a.e.}\right\}
$$
with some $\gamma>0$. Moreover, $F$ is Fr\'{e}chet
differentiable, the Fr\'{e}chet derivative and its adjoint are given by
\begin{align*}
F'(c)h&=-A(c)^{-1}(hu(c)),\\
F'(c)^*w&=-u(c)A(c)^{-1}w,
\end{align*}
where $A(c):H^2\cap H_0^1\mapsto L^2$ is defined by
$A(c)u=-u''+c u$. It has been shown in \cite{HNS96} that if, for
the sought solution $c^\dag$, $|u(c^\dag)|\ge \kappa>0$ on
$[0,1]$, then Assumption \ref{A1} is satisfied in a neighborhood of
$c^\dag$.

In the following we report some numerical results on the method
given by (\ref{2}) and (\ref{DP}) with $g_\a$ defined by (\ref{g20}),
which, in the current context, defines the iterative solutions $\{c_n^\d\}$ by
\begin{align*}\label{6.28.1}
u_{n,0}&=c_n^\d,\\
u_{n, l+1}&=u_{n,l}-F'(c_n^\d)^* \left(F(c_n^\d)-u^\d -F'(c_n^\d)(c_n^\d-u_{n,l})\right), \quad 0\le l\le [1/\a_n]-1,\\
c_{n+1}^\d&= u_{n, [1/\a_n]}.
\end{align*}
and determines the stopping index $n_\d$ by
\begin{equation}\label{6.28.2}
\|F(c_{n_\d}^\d)-u^\d\|\le \tau \d < \|F(c_n^\d)-u^\d\|, \quad 0\le
n<n_\d.
\end{equation}
During the computation, all differential equations are solved
approximately by finite difference method by dividing the interval
$[0, 1]$ into $m+1$ subintervals with equal length $h=1/(m+1)$; we
take $m=100$ in our actual computation.

\subsection*{Example 1}
We estimate $c$ in (\ref{6.1}) by assuming $f(t)=(1+t)(1+t-0.8 \sin(2\pi t))$, $g_0=1$ and
$g_1=2$. If $u(c^\dag)=1+t$, then $c^\dag=1+t-0.8 \sin(2\pi t)$ is the sought
solution. When applying the above method, we take $\a_n=2^{-n}$
and use random noise data $u^\d$ satisfying
$\|u^\d-u(c^\dag)\|_{L^2[0,1]}=\d$ with noise level $\d>0$.
As an initial guess we choose $c_0=1+t$. One can show that $c_0-c^\dag\in {\mathcal
R}(F'(c^\dag)^*)$. Thus, according to Theorem \ref{T1}, the
expected rate of convergence should be $O(\d^{1/2})$.

\begin{table}[ht]
\center{{ Table 1. Numerical results for Example 1 with
$\a_n=2^{-n}$ and three distinct values of $\tau$,
where $n_\d$ is determined by (\ref{6.28.2}),
$error:=\|c_{n_\d}^\d-c^\dag\|_{L^2}$, and $ratio:=error/\d^{1/2}$
} }
\begin{center}
\begin{tabular}{|c|c|c|c|c|c|c|c|c|c|} \hline
      &\multicolumn{3}{c|} {$\tau=1.1$ }&
 \multicolumn{3}{c|}{$\tau=2.0$} & \multicolumn{3}{c|}{$\tau=4.0$}\\
 \cline{2-10}
    $\delta$ & $k_\delta$ & $error$ &
$ratio$ & $k_\delta$ &
$error$ & $ratio$ & $k_\d$ & $error$ & $ratio$ \\
\hline
$10^{-2}$ & $12$ & $4.67e-2$ & $0.47$ & $9$  & $2.90e-1$ & $2.90$ & $1$  & $5.65e-1$ & 5.65\\
$10^{-3}$ & $14$ & $1.47e-2$ & $0.47$ & $12$ & $3.89e-2$ & $1.23$ & $12$ & $3.89e-2$ & 1.23\\
$10^{-4}$ & $16$ & $4.30e-3$ & $0.43$ & $15$ & $5.30e-3$ & $0.53$ & $14$ & $8.70e-3$ & 0.87\\
$10^{-5}$ & $18$ & $1.30e-3$ & $0.40$ & $17$ & $1.80e-3$ & $0.56$ & $16$ & $2.80e-3$ & 0.87\\
$10^{-6}$ & $21$ & $4.45e-4$ & $0.44$ & $19$ & $6.41e-4$ & $0.64$ & $18$ & $1.00e-3$ & $1.03$\\
\hline
\end{tabular}
\end{center}
\end{table}

The numerical result is reported in Table 1.
In order to see the effect of $\tau$ in the
discrepancy principle (\ref{6.28.2}), we consider the three distinct values
$\tau=1.1$, $2$ and $4$. In order to indicate the dependence of
the convergence rates on the noise level, different values of
$\delta$ are selected. The rates in Table 1 coincide with Theorem
\ref{T1} very well. Table 1 indicates also that the absolute error
increases with respect to $\tau$. Thus, in numerical computation,
one should use smaller $\tau$ if possible.

In order to visualize the computed solutions, we plot in Figure \ref{fig1}
the results obtained for $\tau=1.1$ and various values of the noise level $\d$,
where the solid, dashed, and dash-dotted curves denote the exact solution $c^\dag$,
the initial guess $c_0$, and the computed solution $c_{n_\d}^\d$
respectively. It clearly indicates the fast convergence as $\d\rightarrow 0$ as reported in
Table 1.

\begin{figure}[htp]
\centering
  \includegraphics[scale=0.8]{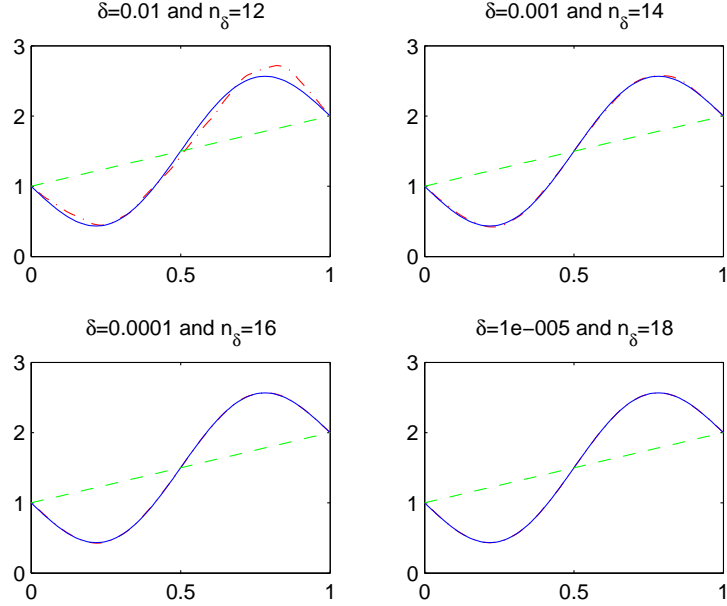}
  \caption{Comparison on the computed and exact solution for Example 1 with $\tau=1.1$ }\label{fig1}
\end{figure}

\subsection*{Example 2}\label{e2}
We repeat Example 1 but with $\tau=1.1$ and the initial guess $c_0=2-t$.
Now $c_0-c^\dag\not\in {\mathcal R} (F'(c^\dag)^*)$, and
in fact $c_0-c^\dag$ has no source-wise representation
$c_0-c^\dag\in {\mathcal R}((F'(c^\dag)^* F'(c^\dag))^\nu)$ with a
good $\nu>0$. However, Theorem \ref{T2} asserts that $\|c_{n_\d}^\d - c^\dag\|_{L^2[0,1]}\rightarrow
0$ as $\delta\rightarrow 0$. Figure \ref{fig2} clearly indicates such convergence
although the convergence speed could be quite slow which is typical for inverse problems.

\begin{figure}[htp]
\centering
  \includegraphics[scale=0.8]{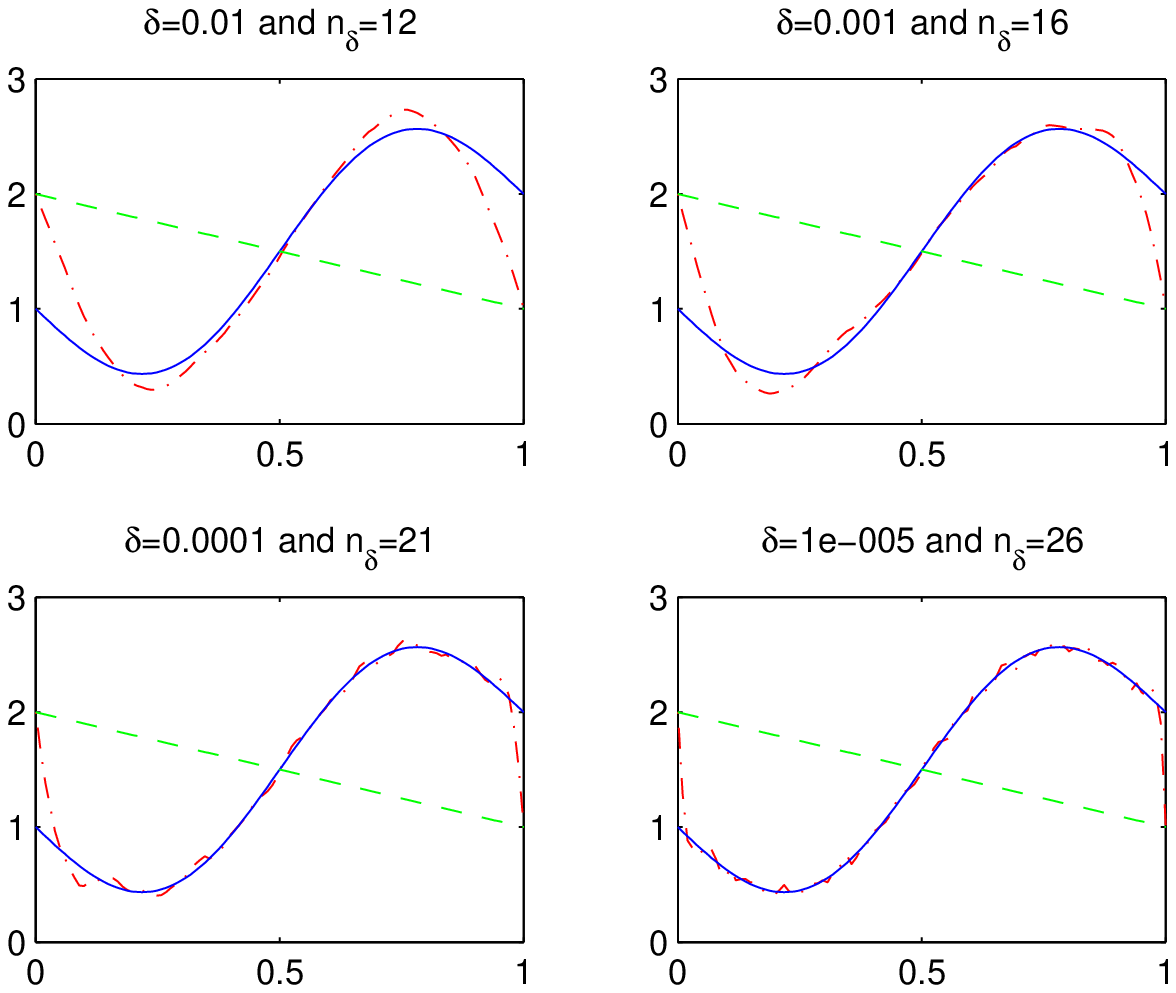}
  \caption{Comparison on the computed and exact solutions for Example 2 with $\tau=1.1$}\label{fig2}
\end{figure}



\end{document}